\documentclass[11pt]{article}

\usepackage{fullpage}
\usepackage{amsmath,amsthm,amsfonts,dsfont}
\usepackage{amssymb,latexsym,graphicx}
\usepackage{palatino}
\usepackage{mathpazo}
\usepackage{stmaryrd}
\usepackage{mathtools}
\usepackage{hyperref}
\usepackage{graphicx}
\usepackage{caption}
\usepackage{url}

\newtheorem{theorem}{Theorem}[section]

\newtheorem*{remark}{Remark}
\newtheorem{proposition}[theorem]{Proposition}
\newtheorem{lemma}[theorem]{Lemma}

\newtheorem{claim}[theorem]{Claim}

\newcommand{\N}{\ensuremath{\mathbb{N}}}

\newcommand{\R}{\ensuremath{\mathbb{R}}}
\newcommand{\Z}{\ensuremath{\mathbb{Z}}}

 \newcommand{\eps}{\varepsilon}
\renewcommand{\epsilon}{\varepsilon}

\DeclareMathOperator*{\expect}{\mathbb{E}}

\newcommand{\pois}{\mathrm{Pois}}
\newcommand{\Pois}{\pois}

\newcommand{\pr}[2]{({#1, #2})}
\makeatletter
\def\imod#1{\allowbreak\mkern8mu({\operator@font mod}\,\,#1)}
\makeatother

\newcommand{\puv}{p_{u,v}}
\newcommand{\puu}{p_{u,u}}
\newcommand{\pvv}{p_{v,v}}
\newcommand{\ruv}{r_{u,v}}
\newcommand{\rxy}{r_{x,y}}
\newcommand{\pxy}{p_{x,y}}

\newcommand{\zerovec}{{\mathbf{0}}}

\DeclarePairedDelimiter\parens{(}{)}

\newif\ifnotes\notesfalse

\ifnotes
\usepackage{color}
\definecolor{mygrey}{gray}{0.50}
\newcommand{\notename}[2]{{\textcolor{mygrey}{\footnotesize{\bf (#1:} {#2}{\bf ) }}}}
\newcommand{\noteswarning}{{\begin{center} {\Large WARNING: NOTES ON}\end{center}}}

\else

\newcommand{\notename}[2]{{}}
\newcommand{\noteswarning}{{}}

\fi

\begin{document}

\title{A counterexample to monotonicity of relative mass in random walks}


\author{
Oded Regev\thanks{Courant Institute of Mathematical Sciences, New York
 University.}
~\thanks{Supported by the Simons Collaboration on Algorithms and Geometry and by the National Science Foundation (NSF) under Grant No.~CCF-1320188. Any opinions, findings, and conclusions or recommendations expressed in this material are those of the authors and do not necessarily reflect the views of the NSF.}\\
\and
Igor Shinkar\footnotemark[1]
~\thanks{
    Research supported by NSF grants CCF 1422159, 1061938, 0832795  and Simons Collaboration on Algorithms and Geometry grant.
}
}
\date{}
\maketitle

\noteswarning

%

\begin{abstract}
For a finite undirected graph $G = (V,E)$, let $\puv(t)$ denote the
probability that a continuous-time random walk
starting at vertex $u$ is in $v$ at time $t$.
In this note we give an example of a Cayley graph $G$ and two vertices
$u,v \in G$ for which the function
\[
	\ruv(t) = \frac{\puv(t)}{\puu(t)} \qquad t \geq 0
\]
is not monotonically non-decreasing.
This answers a question asked by Peres in 2013.
\end{abstract}

\section{Introduction}

Let $G = (V,E)$ be a finite undirected regular graph.
Let $\puv(t)$ denote the probability that a continuous-time random walk
starting at vertex $u$ is in $v$ at time $t$.
In this note we are interested in the function
\[
	\ruv(t) = \frac{\puv(t)}{\puu(t)} \qquad t \geq 0 \; .
\]
Clearly, in regular connected graphs for any $u \neq v$, we have $\ruv(0)=0$ and $\lim_{t \to \infty}\ruv(t) = 1$.
One might wonder if the function is monotonically non-decreasing.
It is not difficult to see that there are regular graphs for which this is \emph{not}
the case.
In fact, there are regular graphs such that $\ruv(t) > 1$ for some vertices $u,v$ and time $t$;
in particular, $\ruv(t)$ is not monotonically non-decreasing.
We give an example of such a graph in Appendix~\ref{sec:appendix}.
We thank Jeff Cheeger~\cite{Cheeger} for pointing this out to us.

For vertex-transitive graphs, however, it holds that $\ruv(t) \le 1$ for all vertices $u,v$ and all $t \geq 0$.
Indeed, using Cauchy-Schwarz and the reversibility of the walk,
\begin{align*}
	\puv(t) & =
	\sum_{w \in V} p_{u,w}(t/2) p_{w,v}(t/2)\\
	& \leq \parens[\Big]{\sum_{w} p_{u,w}(t/2)^2}^{1/2} \cdot \parens[\Big]{\sum_{w} p_{w,v}(t/2)^2}^{1/2} \\
	& =  \puu(t)^{1/2} \cdot \pvv(t)^{1/2}
	 =  \puu(t) \; .
\end{align*}
This motivates the following question, asked in 2013 by Peres~\cite{Peres}:
\begin{quote}
Is the function $\ruv$ monotonically non-decreasing
in $t$ for all vertex-transitive graphs and all vertices $u,v$?
\end{quote}
More recently, a special case of that question was asked independently by Price~\cite{Price}.
Namely, Price asked whether for Brownian motion on flat tori (i.e., on $\R^n$ modulo a lattice),
it holds that for any point $x$, the density at $x$ divided by the density at the starting point
$x_0$ is monotonically non-decreasing in time. This would follow from a positive answer to Peres's question
through a limit argument. Price gave a positive answer to his question for the case of a cycle ($n=1$)
and recently, a positive answer for arbitrary flat tori was found~\cite{RSD15}. This can be seen
as further evidence for a positive answer to Peres's question.

In this note we give a negative answer to Peres's question.
In fact, we do so through a Cayley graph.

\begin{theorem}\label{thm:non-monotone}
	There exists a Cayley graph $G = (V,E)$ and two vertices $u,v \in V$
	such that the function $\ruv$ is not monotonically non-decreasing.
\end{theorem}

One remaining open question is whether $\ruv$ is monotonically non-decreasing
for \emph{Abelian} Cayley graphs. The positive result of~\cite{RSD15} is
a special case of that.

\subsection{Some basic facts about continuous-time random walks}
Given a weighted finite graph $G=(V,E)$ with weight function $w :E \to \R_+$
a continuous-time random walk $X=(X_t)_{t \geq 0}$ on $G$
is defined by its heat kernel $H_t$, that at time $t>0$ is equal to
\[
	H_t = e^{-t \cdot L},
\]
where $L$ is the Laplacian matrix of $G$ given by $L_{u,v} = -w(u,v)$ for $u \neq v$,
and $L_{u,u} = \sum_{v} w(u,v)$.
As a result, for a random walk $X$ starting at a vertex $u$
the probability that $X$ is in $v$ at time $t$ is equal to
$\puv(t) := H_t(u,v)$.
When $G$ is a $d$-regular
unweighted simple graph,
we think of the edges as all having weight $1/d$, in which
case the Laplacian of $G$ is given by
\[
	L_{u,v} = \begin{cases}
			-1/d & \mbox{if $(u,v) \in E$} \\
			   1 & \mbox{if $u = v$} \\
			   0 & \mbox{otherwise.}
			   \end{cases}
\]

In this note we consider only vertex-transitive graphs, for which the sum $\sum_v w(u,v)$ is the same
for all vertices $u$ of the graph. Note that we do not insist that
this sum is equal to 1,
though this can be achieved by normalizing $L$, which corresponds to changing the speed of the
random walk. For basic facts about continuous-time random walks see, e.g.,~\cite{MCMTbook}.

If $G$ is a weighted Cayley graph with a generating set $S$ and a weight function $w:S \to \R_+$,
then a continuous-time random walk $X=(X_t)_{t \geq 0}$ on $G$ is described
by mutually independent Poisson processes of rate $w(g)$ for each group generator $g \in S$,
where each process indicates the times when $X$ jumps along the corresponding edge.

\section{Non-monotonicity of time spent at the origin in the hypercube graph}

For an integer $d \geq 1$ denote by $Q_d$ the $d$-dimensional hypercube graph.
The vertices of $Q_d$ are $\{0,1\}^d$ and there is an edge between two vertices
$u$ and $v$ if and only if they differ in exactly one coordinate.
Let $X = (X_t)_{t > 0}$ be a continuous-time random walk on $Q_d$
starting at the origin, denoted by $\zerovec = (0,\dots,0) \in \{0,1\}^d$.
Denote by $C_d(t)$ the expected time spent at the origin until time $t$,
conditioned on the event that $X_t = \zerovec$.
That is,
\[
	C_d(t) = \int_{0}^t\Pr[ X_s = \zerovec | X_t = \zerovec] ds.
\]
In this section we show that for $d$ sufficiently large
$C_d(t)$ is not monotonically non-decreasing.

\begin{lemma}\label{lemma:non-monotonicity}
	Let $d \in \N$ be sufficiently large.
	Then, there are some $t_1 < t_2$ such that $C_d(t_1) > C_d(t_2)$,
	and in particular, the function $C_d$ is not monotonically non-decreasing in $t$.
\end{lemma}

\begin{remark}
Numerically, one can see that the function $C_d$ is not monotone for $d \ge 5$.
See Figure~\ref{fig:C_5}. Since $C_d$ has a closed form expression (as can be
seen from the calculations below), one can probably show
non-monotonicity directly for $C_5$ by analyzing the function,
though doing so would likely be messy and not too illuminating.

  \begin{figure}[htbp]
    \centering
    \includegraphics[height=60mm]{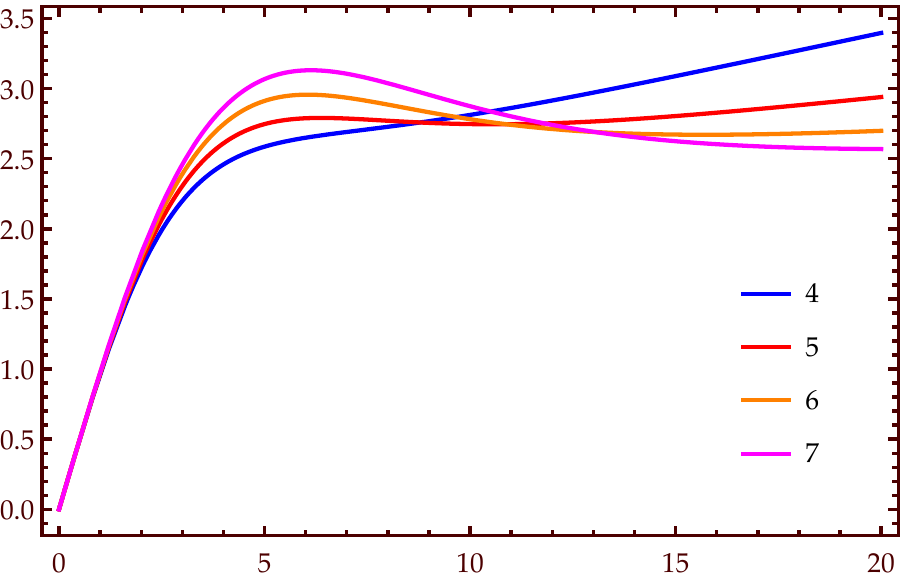}
    \caption{$C_d(t)$ for $d=4,5,6,7$ (from top right to bottom right).}
    \label{fig:C_5}
  \end{figure}
\end{remark}

Before proving Lemma~\ref{lemma:non-monotonicity} we prove the following claim.
\begin{claim}\label{claim:return prob}
	Let $d \ge 1$, and let $Q_d$ be the $d$-dimensional hypercube graph.
	Let $X = (X_t)_{t > 0}$ be a continuous-time random walk on $Q_d$ starting at $\zerovec$.
	Then,
	\[
		\Pr[X_t = \zerovec] = \parens[\Big]{\frac{1 + e^{-2t/d}}{2}}^d.
	\]
\end{claim}
\begin{proof}
	Since $X$ moves in each coordinate with rate $1/d$,
	it follows that for each $i \in [d]$ the number of
	steps in direction $i$ up to time $t$ is distributed like $\Pois(t/d)$.
	Therefore,
	\[
	\Pr[(X_t)_i = 0] = \Pr[\Pois(t/d) \text{ is even}] = (1 + e^{-2t/d})/2 , \]
	where we used that the probability that $\Pois(\lambda)$ is even is
	\[
		\Pr[\Pois(\lambda) \text{ is even}]
		= e^{-\lambda} \cdot \sum_{j \text{ even}} \frac{\lambda^j}{j!}
		= e^{-\lambda} \cdot \frac{1}{2} \parens[\Big]{\sum_{j=0}^\infty \frac{\lambda^j}{j!} + \sum_{j=0}^\infty \frac{(-\lambda)^j}{j!}}
		= e^{-\lambda} \cdot \frac{1}{2} \parens[\Big]{e^{\lambda} + e^{-\lambda}} \; .
	\]	
	Since the coordinates of $X$ move independently	the result follows.
\end{proof}

We now prove Lemma~\ref{lemma:non-monotonicity}.
\begin{proof}[Proof of Lemma~\ref{lemma:non-monotonicity}]
	We show below that for all $d \ge 1$, it holds that
	\begin{enumerate}
		\item $C_d(\sqrt{d}) \ge e^{-1}\sqrt{d}$,
		\item $C_d(d) \le 6$.
	\end{enumerate}
	This clearly proves the lemma for $d$ sufficiently large.
	
	To prove Item 1, we show that if a walk starting
    from the origin is at the origin at time $\sqrt{d}$,
	then 
	with constant probability it stayed at the origin throughout that time interval.
    Intuitively, this is because the
	probability of a coordinate flipping twice during that time
	is of order only $1/d$ and so with constant probability 
	none of the $d$ coordinates flips.
    In more detail, by Claim~\ref{claim:return prob},
	\[
		\Pr[X_{\sqrt{d}} = \zerovec] =
		\parens[\Big]{\frac{1 + e^{-2/\sqrt{d}}}{2}}^d
		\leq \parens[\Big]{1 - \frac{1}{\sqrt{d}} + \frac{1}{d}}^d
		= \parens[\Big]{1 - \frac{\sqrt{d}-1}{d}}^d
		\leq e^{-\sqrt{d}+1} \; ,
	\]
	where we used the inequality
	$e^{-x} \le 1-x+x^2/2$ valid for all $x \ge 0$.
	On the other hand, by definition of a continuous-time random walk the
	probability that $X$ stays in $\zerovec$ during the entire time interval
	$[0,\sqrt{d}]$ is equal to $\Pr[X_{[0,\sqrt{d}]} \equiv \zerovec] =e^{-\sqrt{d}}$.
	Therefore,
	\[
		\Pr[X_{[0,\sqrt{d}]} \equiv \zerovec | X_{\sqrt{d}} = \zerovec] \geq e^{-1},
	\]
	and hence the expected time spent at the origin conditioned on $X_{\sqrt{d}} = \zerovec$
	is as claimed in Item 1.
	
	We next prove Item 2. Intuitively, here there is enough time for coordinates
	to flip twice, and only a very small part of the time will be
	spent at the origin.
    By definition of $C_d$ and Claim~\ref{claim:return prob} we have
	\begin{align*}
		C_d(t) &= \int_{0}^t \frac{\Pr[ X_s = \zerovec] \cdot \Pr[X_{t-s} = \zerovec]}{\Pr[X_t = \zerovec]} ds \\
		&= \int_{0}^t (h_d(t,s))^d ds,
	\end{align*}
	where
	\[
	h_d(t,s) =
	\frac{(1 + e^{-2s/d})(1 + e^{-2(t-s)/d})}{2(1 + e^{-2t/d})}
	=
	\frac{1 + e^{-2s/d} + e^{-2(t-s)/d} + e^{-2t/d}}{2(1 + e^{-2t/d})}
	.\]
	Since $h_d(t,s)$ is convex as a function of $s$,
	for all $0 \le s \le t/2$ we have $h_d(t,s) \le \ell(s)$
	where $\ell$ is the unique linear function
	satisfying $\ell(0)=h_d(t,0)$ and $\ell(t/2) = h_d(t,t/2)$.
	Therefore, taking $t=d$, we get
		\[
			h_d(d,s)
			\leq \ell(s/d)
			= 1 - \frac{cs}{d} \; ,
		\]	
	where $c = (1 -e^{-1})^2/(1 + e^{-2})$.
	Noting that $h_d(t,s) = h_d(t,t-s)$, we get
	\[
		C_d(d) = \int_{0}^d (h_d(d,s))^d ds
		 = 2 \int_{0}^{d/2} (h_d(d,s))^d ds
		< 2 \int_{0}^{d/2} \parens[\big]{1 -\frac{c s}{d}}^d ds
		< 2 \int_{0}^{d/2} e^{-cs} ds
		\leq \frac{2}{c} \; .
	\]
	This completes the proof of Lemma~\ref{lemma:non-monotonicity}.
\end{proof}

\section{Proof of Theorem~\ref{thm:non-monotone}}\label{sec:proof of thm}

In this section we prove Theorem~\ref{thm:non-monotone}.
We first
give a proof for a weighted graph, and then remark
on how to convert it into an unweighted graph.
For $d \in \N$ sufficiently large
we define the weighted graph $G$ to be the \emph{lamplighter graph} on $Q_d$,
whose edges corresponding to steps on $Q_d$ are of weight $1/d$,
and edges corresponding to toggling a lamp are of weight $\eps$,
for some $\eps>0$ sufficiently small that depends on $d$
and $t_1,t_2$ from Lemma~\ref{lemma:non-monotonicity}.

In more detail, the weighted lamplighter graph $G$ is described by
placing a lamp at each vertex of $Q_d$ and a lamplighter walking on $Q_d$.
A vertex of $G$ is described by the location $x \in \{0,1\}^d$ of the lamplighter,
and a configuration $f : \{0,1\}^d \to \{0,1\}$ indicating which lamps are currently on.
In each step the lamplighter either makes a step in the graph $Q_d$
or toggles the state of the lamp in the current vertex.
More formally, we have an edge between $(x,f)$ and $(y,g)$ if and only if either
\begin{enumerate}
\item $(x,y) \in E_d$ and $f = g$ (this corresponds to a step in $Q_d$) or
\item $x = y$ and $f$ and $g$ differ on the input $x$ and are equal on all other inputs
(this corresponds to toggling a lamp at $x$).
\end{enumerate}
The weights of the edges of the first type are $1/d$,
and the edges of the second type are of weight $\eps$.
Thus, in a random walk on $G$,
the steps of the lamplighter are distributed as in a random walk on $Q_d$,
and the number of times the lamps are toggled in a time interval of length $T$
is distributed like $\Pois(\eps T)$ independently of the the lamplighter's walk.
It is well known that the lamplighter graph is a Cayley graph (see, e.g.,~\cite{PeresRevelle04}).
%

%

Let $u$ be the vertex in $G$ corresponding to the lamplighter being at the origin with all lights off.
Let $v$ be the vertex in $G$ corresponding to the lamplighter being at the origin with the light at the origin being on, and all other lights off.
We show below that $\ruv$ is not monotonically non-decreasing.
More specifically, we show that $\ruv(t_1) > \ruv(t_2)$,
where $t_1 < t_2$ are from Lemma~\ref{lemma:non-monotonicity}.

Let $X=(X_t)_{t \geq 0}$ be a continuous-time random walk on $G$ starting at $X_0 = u$.
Denote by $Y_t$ the number of times a toggle occurred during the time interval $[0,t]$.
Denote by $Z = (Z_t)_{t \geq 0}$ the trajectory of the lamplighter,
i.e., the projection of $X$ to the first coordinate.
Note that by definition $Z$ is a continuous-time random walk on $Q_d$,
and that $Z$ is independent of $Y_t$.

\begin{claim}\label{claim:puu puv}
	Let $u,v \in V$ be as above. Then, for all $t>0$ it holds that
	\begin{equation}\label{eq:puu}
		0 \leq \puu(t) - e^{-\eps t} \cdot \Pr[Z_{t}=\zerovec]  \leq \eps^2 t^2 \; ,
	\end{equation}
	and
	\begin{equation}\label{eq:puv}
		0 \leq \puv(t) - \eps e^{-\eps t} \cdot C_d(t) \cdot \Pr[Z_{t}=\zerovec] \leq \eps^2 t^2 \; .
	\end{equation}
\end{claim}

Using the claim, 
\[
	\ruv(t) = \frac{\puv(t)}{\puu(t)} = \eps \cdot C_d(t) \pm O(\eps^2),
\]
where $O(\cdot)$ hides a constant that depends on $d$ and $t$.
In particular, for $t_1 < t_2$
from Lemma~\ref{lemma:non-monotonicity}, and $\eps>0$ sufficiently small
we get that $\ruv(t_1) > \ruv(t_2)$, which proves Theorem~\ref{thm:non-monotone}.

Intuitively, \eqref{eq:puu} holds because the probability
of toggling a lamp twice is very small,
and hence $\puu(t)$ is approximately equal to the probability
that no lamp has changed its state multiplied by the probability 
that a random walk on $Q_d$ will be at the origin at time $t$.
The intuition for~\eqref{eq:puv} is that
in order to get from $u$ to $v$, 
in addition to getting back to the origin,
the lamplighter must toggle the switch
while being at the origin, and the probability of that is roughly $\eps \cdot C_d(t)$.

\begin{proof}[Proof of Claim~\ref{claim:puu puv}]
For $\puu$ we have
\[
  p_{u,u}  =  \Pr[X_t = u \wedge Y_t=0] + \Pr[X_t=u \wedge Y_t \ge 2] \; .
\]
Since $Y_t$ is distributed  like $\Pois(\eps t)$, the second term satisfies
\[
  0 \le \Pr[X_t=u \wedge Y_t \ge 2]
	  \le \Pr[Y_t \ge 2]
		\le \eps^2 t^2 \; ,
\]
and for the first term, by independence between $Y_t$ and $Z_t$ we have
\[
  \Pr[X_t = u \wedge Y_t=0]  = \Pr[Z_t = \zerovec \wedge Y_t=0] =
  e^{-\eps t} \cdot \Pr[Z_{t}=\zerovec] \; ,
\]
proving~\eqref{eq:puu}.

For $\puv$ we similarly have
\[
	\puv(t)= \Pr[X_t = v \wedge Y_t = 1] + \Pr[X_t = v \wedge Y_t \geq 2].
\]
As above, the second term is at most $\eps^2 t^2$.
For the first term,
let $E_t$ be the event that $Y_t = 1$, and the unique lamp that is on at time $t$ is the lamp at the origin.
Denote by $T_0$ the time spent by $Z$ at the origin in the time interval $[0,t]$.
Then, conditioning on $Z$, the event $E_t$ holds if and only if a unique switch happened during $T_0$ time,
and zero switches in the remaining time.
Therefore, by independence of a Poisson process in disjoint intervals
\[
	\Pr[E_t |Z]
	= \Pr[\Pois(\eps T_0) = 1 |Z] \cdot \Pr [\Pois(\eps (t - T_0)) = 0|Z]
	= \eps T_0 \cdot e^{-\eps T_0} \cdot e^{-\eps (t-T_0)}
	= \eps e^{-\eps t} \cdot T_0.
\]
This implies that
\[
\Pr[X_t = v \wedge Y_t = 1]
= \Pr [E_t | Z_t = \zerovec] \cdot \Pr[Z_t = \zerovec]
= \eps e^{-\eps t} \cdot \expect[T_0 | Z_t = \zerovec] \cdot \Pr[Z_t = \zerovec].
\]
Therefore, since $C_d(t) = \expect[T_0 | Z_t = \zerovec]$ we get~\eqref{eq:puv}, and the claim follows.
\end{proof}

\paragraph{Converting $G$ into an unweighted graph.}
	Below we show how to convert a weighted Cayley graph $G$ into an unweighted one,
	while preserving the property in Theorem~\ref{thm:non-monotone}.
	Let $(G,S_G)$ be a weighted Cayley graph with the generating set $S_G = \{g_1,\dots, g_k\}$,
	and suppose that the weights $w:S_G \to \R_+$ are integers for all $g \in S_G$.
	For $N \in \N$ sufficiently large define the graph $H$ by
	replacing each vertex $v \in G$ with an $N$-clique $\{\pr{v}{i} : i \in \Z_N\}$,
	and replacing each edge $(u,ug)$ in $G$ of weight $w(g)$ with $w(g)$ perfect
	matchings $\{\pr{u}{i},\pr{ug}{i+j} : i \in \Z_N\}_{j=1}^{w(g)}$.

	Formally, the vertices of the graph $H$ are $G \times \Z_N = \{\pr{v}{i} : v \in G, i \in \Z_N\}$,
	and the set of generators $S_H$ for the Cayley graph on $H$ is given by
	\[
		S_H = \{\pr{0}{i} : i \in \Z_N \setminus \{0\} \}
		\bigcup \cup_{g \in S_G} \{ \pr{g}{j}: j \in \{1,\dots,w(g)\} \}.
	\]

	Note that the projection of a continuous-time random walk on $H$
	to the first coordinate is a random walk on $G$, slowed down by $\deg(H)$.
	Moreover, assuming $N$ is larger than, say, $\sum_{g \in S_G} w(g)$,
	after constant time the two coordinates become close to independent
	with the 	second coordinate being uniform.
	Therefore, if $u,v$ are vertices in $G$, and $x = \pr{u}{0}, y = \pr{v}{0}$
	are the corresponding vertices in $H$,
	then for any time $t > 0$ and $t' = \deg(H) \cdot t$
	it holds that $\pxy(t') = \frac{1}{N}(\puv(t) \pm o_N(1))$
	and hence $\rxy(t') = \ruv(t') \pm o_N(1)$.
	
	For the graph $G$ given in the proof of Theorem~\ref{thm:non-monotone}
	above, we may assume that $1/\eps$ is an integer,
	and so, by multiplying all weights by $d/\eps$ we get
	a Cayley graph with integer weights.
	Hence, by applying the foregoing transformation
	we get a simple unweighted Cayley graph $H$ for which $\ruv$ is not
	monotonically non-decreasing for some $u,v \in H$.

\appendix
\section{Appendix: A counterexample in a regular non-transitive graph}\label{sec:appendix}

Below we give a simple example of a regular non-transitive graph
such that $\ruv(t) > 1$ for some vertices $u,v$ and some time $t$;
in particular, $\ruv(t)$ is not monotonically non-decreasing,
since $\ruv(t) \to 1$ as $t \to \infty$.
We thank Jeff Cheeger~\cite{Cheeger} for pointing this out to us.

\begin{proposition}\label{prop:cheeger}
Let $L$ be the Laplacian of a regular graph on vertex set $V$.
Denote its eigenvalues by
$0 = \lambda_1 \leq \lambda_2 \leq \dots \leq \lambda_{|V|}$
and by $f_i \in \R^V$ the corresponding normalized eigenvectors.
Suppose that $0 < \lambda_2 < \lambda_3$,
and that $f_2$ is such that $f_2(v) > f_2(u) > 0$ for some vertices $u,v$.
Then, there is some $t>0$ such that $\ruv(t) > 1$.
\end{proposition}

\begin{proof}
	Let $\pi_u \in \R^V$ be the vector with $\pi_u(u) = 1$
	and $\pi_u(u') = 0$ for all $u' \neq u$.
	Writing $\pi_u = \sum \alpha_i f_i$ for $\alpha_i = \langle \pi_u, f_i \rangle = f_i(u)$,
	for all $w \in V$ we have
	\[
		e^{-t L} \pi_u(w)
	= \sum_{i=1}^{|V|} e^{-t \lambda_i} \alpha_i \cdot f_i(w)
	= c + e^{-\lambda_2 t} f_2(u) f_2(w) + O(e^{-\lambda_3 t}),
	\]
	where $O()$ hides some constants that may depend on the graph, but not on $t$,
	and $c = \alpha_1 \cdot f_1(w)$ is independent of $w$ since $f_1$ is a constant function.
	Using the facts that $f_2(v) > f_2(u) > 0$
	and $\lambda_3 > \lambda_2$, it follows that
	for sufficiently large $t$,
	 \[
	\ruv(t) = \frac{e^{-t L}\pi_u(v)}{e^{-t L} \pi_u(u)} > 1 \; ,
	\]
	as desired.
\end{proof}

  \begin{figure}[htbp]
    \centering
    \includegraphics[height=60mm]{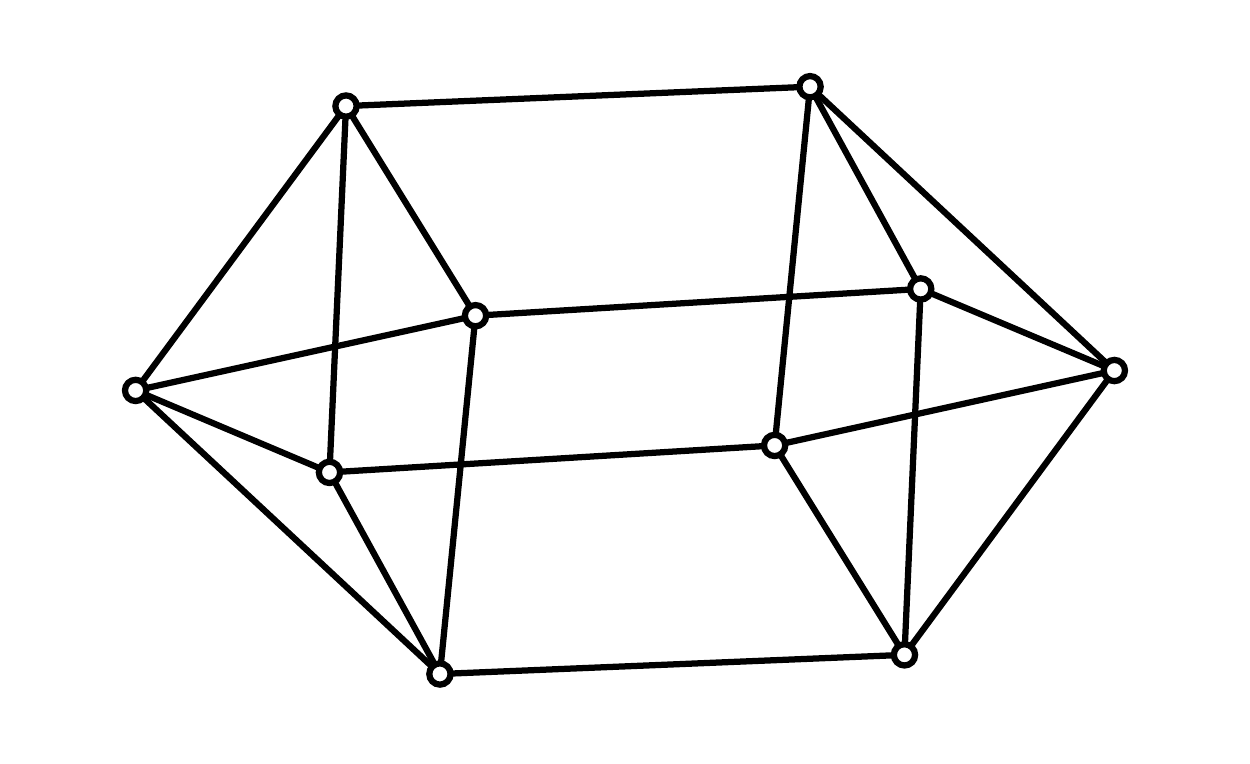}
    \caption{A cube with two square pyramids attached.}
    \label{fig:Oct}
  \end{figure}

Graphs satisfying the constraints in Proposition~\ref{prop:cheeger} are in
abundance. As a concrete example, consider the 4-regular graph on 10 vertices shown
in Figure~\ref{fig:Oct}.
Using Mathematica, we see that the second eigenvalue of the Laplacian of this
graph is $\lambda_2 = \frac{1}{8}(7 - \sqrt{17}) \approx 0.36$,
and it is a simple eigenvalue.
The corresponding (non-normalized) eigenvector with vertices ordered from left to right is
$(c,1,1,1,1,-1,-1,-1,-1,-c)$,
where $c = 3-\frac{1}{2}(7 - \sqrt{17}) \approx 1.56$.
In particular, Proposition~\ref{prop:cheeger} is applicable to this graph.

\paragraph{Acknowledgements}
We thank Eyal Lubetzky and Yuval Peres for helpful comments.
We also thank an anonymous referee for useful comments, and
for pointing out that our construction can be seen as a
lamplighter walk.
\bibliographystyle{alphaabbrvprelim}
\bibliography{monotonicity}

\end{document}